\documentclass[11pt]{amsart}
\usepackage[a4paper,margin=1.12in]{geometry}
\usepackage{amsmath,amssymb,amsthm,mathtools,microtype}
\usepackage[colorlinks=true,linkcolor=blue,citecolor=blue,urlcolor=blue]{hyperref}
  
\newtheorem{theorem}{Theorem}[section]
\newtheorem{proposition}[theorem]{Proposition}
\newtheorem{lemma}[theorem]{Lemma}
\newtheorem{corollary}[theorem]{Corollary}
\theoremstyle{definition}

\theoremstyle{remark}
\newtheorem{remark}[theorem]{Remark}

\newcommand{\D}{\mathbb D}
\newcommand{\T}{\mathbb T}
\newcommand{\R}{\mathbb R}
\newcommand{\C}{\mathbb C}
\newcommand{\N}{\mathbb N}
\newcommand{\Hh}{\mathcal H}
\newcommand{\dd}{\,\mathrm d}
\newcommand{\e}{\mathrm e}
\newcommand{\SU}{\mathrm{SU}(1,1)}
\newcommand{\norm}[1]{\left\|#1\right\|}
\newcommand{\ip}[2]{\left\langle #1,#2\right\rangle}
\newcommand{\supp}{\operatorname{supp}}
\newcommand{\Rea}{\operatorname{Re}}

\title[The nonlinear Hausdorff--Young inequality]{The nonlinear Hausdorff--Young inequality}

\author[D. Suragan]{Durvudkhan Suragan}
\address{
	Durvudkhan Suragan:
	\endgraf
	Department of Mathematics
	\endgraf
	Nazarbayev University
	\endgraf
	Astana 010000
	\endgraf
	Kazakhstan
	\endgraf
	\textit{E-mail address:} \textrm{durvudkhan.suragan@nu.edu.kz}
}


\begin{document}

\begin{abstract}
We prove the constant-one discrete nonlinear Hausdorff--Young inequality. As a consequence, by a discrete-to-continuous limiting argument,
we obtain
\[
 \big\|(\log|a_f|^2)^{1/2}\big\|_{L^{p'}(\R)}
 \le \|f\|_{L^p(\R)},\qquad 1\le p<2,
\]
for all $f\in L^p(\R)$, where $a_f$ denotes the transmission coefficient of the nonlinear Fourier transform.
In particular, this resolves the Muscalu--Tao--Thiele uniformity problem. The proof uncovers a hidden Hilbert-space structure that reduces the nonlinear inequality to classical interpolation.
\end{abstract}

\maketitle

\section{Introduction and main results}

\subsection{About the problem}
The nonlinear Fourier transform considered here is the $\SU$ scattering
transform attached to the one-dimensional Dirac system.  We normalize the
continuous transform as follows.  For a compactly supported complex-valued
potential $f$ and spectral parameter $\xi\in\R$, let
\begin{equation}\label{eq:dirac-left}
 \partial_x
 \binom{a}{b}
 =
 \begin{pmatrix}
  0&\overline{f(x)}\e^{2\pi i x\xi}\\
  f(x)\e^{-2\pi i x\xi}&0
 \end{pmatrix}
 \binom{a}{b},
 \qquad
 \binom{a}{b}(-\infty,\xi)=\binom10.
\end{equation}
The terminal values are denoted by $a_f(\xi)$ and $b_f(\xi)$.  Since the
coefficient matrix in \eqref{eq:dirac-left} belongs to the Lie algebra
$\mathfrak{su}(1,1)$ of the matrix group $\SU$, the indefinite Hermitian form is conserved and
\begin{equation}\label{eq:SU-conservation}
 |a_f(\xi)|^2-|b_f(\xi)|^2=1.
\end{equation}
We use
\begin{equation}\label{eq:Hf-def}
 H_f(\xi):=\bigl(\log |a_f(\xi)|^2\bigr)^{1/2}.
\end{equation}
This is the standard logarithmic nonlinear Fourier amplitude. For details, we refer to the
lecture notes of Tao and Thiele \cite{TT} and the more recent works
\cite{Kovac12,OeS18,KOR19,KOR22}.

Two endpoint facts (see, e.g. \cite{Thiele04}) suggest a nonlinear Hausdorff--Young principle.  The
nonlinear Plancherel identity reads
\begin{equation}\label{eq:nonlin-plancherel}
 \norm{H_f}_{L^2(\R)}=\norm{f}_{L^2(\R)},
\end{equation}
and the nonlinear Riemann--Lebesgue inequality is
\begin{equation}\label{eq:L1Linf}
 \norm{H_f}_{L^\infty(\R)}\le \norm{f}_{L^1(\R)}.
\end{equation}
For each fixed $1<p<2$, Christ--Kiselev methods \cite{CK1,CK2} imply a nonlinear
Hausdorff--Young estimate
\begin{equation}\label{eq:knownNLHY}
 \norm{H_f}_{L^{p'}(\R)}\le C_p\norm{f}_{L^p(\R)},\quad p':=\frac{p}{p-1},
\end{equation}
with $C_p<\infty$. The known proofs, however, do not
interpolate the two endpoint statements and their constants deteriorate as
$p\uparrow2$. Throughout this paper, as usual, $p'$ denotes the conjugate exponent, with the
convention $1'=\infty$.

The question whether the constants in the nonlinear
Hausdorff--Young inequality \eqref{eq:knownNLHY} can be chosen uniformly as
$p\uparrow2$ was raised by Muscalu, Tao, and Thiele
\cite{MTTcantor}; see also \cite[Conjecture~2]{Thiele04}, \cite[Conjecture~1.1]{Kovac12}
and \cite[Question~2.1]{OeS17} for explicit
formulations of the problem.

The same question arises in the usual discrete $\SU$ Fourier product
\cite{MTTcantor,Kovac12,KOR22}.  Kova\v c \cite{Kovac12} proved uniformity for the Cantor
model.  Kova\v c, Oliveira e Silva and Rup\v ci\'c \cite{KOR19} obtained sharp
perturbative results for small continuous potentials, while
Oliveira e Silva \cite{OeS18} established a discrete variational inequality.  In the discrete model, Kova\v c, Oliveira e Silva and Rup\v ci\'c
\cite{KOR22} recorded both the uniform-constant question
and the sharper constant-$1$ inequality as (unsolved) open problems.

In this paper, we introduce a new approach.  Muscalu, Tao, and Thiele \cite{MTTcounter} showed that
an endpoint strategy based on estimating the individual multilinear terms
of the scattering expansion cannot succeed in the expected form.  The proof below avoids the multilinear expansion
altogether.  The logarithmic energy is first represented as a Hilbert-space
norm, and the full nonlinear Schur evolution becomes an affine unitary
cocycle.  Orthogonality is then recovered before taking the Hilbert norm.

The endpoint $p\uparrow2$ is also related to limiting versions of the
classical Hausdorff--Young inequality.  Nursultanov, Ghorbanalizadeh and the author
recently obtained sharp Herz--Bochkarev-type limiting estimates using
Lorentz spaces \cite{NGS25}.  Our result is of a different
nature, but it identifies a mechanism that remains stable at the Plancherel endpoint
(see~\cite{NGS25} also for a brief prehistory).

Connections between nonlinear Fourier analysis, Schur algorithms, and
quantum signal processing have developed rapidly in recent years.  In
particular, Alexis, Mnatsakanyan and Thiele \cite{AMT24} identified a direct link between
quantum signal processing and nonlinear Fourier analysis, and
Alexis, Lin, Mnatsakanyan, Thiele and Wang \cite{ALMTW26} subsequently established an infinite
quantum signal processing theory for arbitrary Szeg\H{o} functions.  These developments provide further evidence that the
nonlinear Fourier product is a natural analytic object beyond inverse
scattering.  

\subsection{Main results}
Let $F=(F_n)_{n\in\mathbb Z}$ have finite support and satisfy $|F_n|<1$.
Set
\begin{equation}\label{eq:ABdef}
 A_n=(1-|F_n|^2)^{-1/2},\qquad B_n=A_nF_n.
\end{equation}
For $t\in\T:=\R/\mathbb Z$ define the ordered product
\begin{equation}\label{eq:discrete-product}
 \begin{pmatrix}a(t)&b(t)\\\overline{b(t)}&\overline{a(t)}\end{pmatrix}
 =\mathop{\prod}^{\longrightarrow}_{n\in\supp F}
 \begin{pmatrix}
 A_n&B_n\e^{2\pi int}\\
 \overline{B_n}\e^{-2\pi int}&A_n
 \end{pmatrix},
\end{equation}
where multiplication is from left to right as $n$ increases.  This is the
convention of \cite{KOR22}. Its relation with the Schur algorithm and the theory of orthogonal polynomials on the unit circle (OPUC)
is classical; see Simon \cite{Simon1,Simon2}, as well as Tao and Thiele \cite{TT}.
Verblunsky's identity takes the form
\begin{equation}\label{eq:Verblunsky}
 \int_\T\log|a(t)|^2\dd t
 =\sum_{n\in\mathbb Z}\log A_n^2.
\end{equation}

Our first theorem settles the constant-one discrete nonlinear Hausdorff--Young problem formulated in \cite[(1.4)]{KOR22}.

\begin{theorem}\label{thm:discrete}
Let $F$ be finitely supported and $|F_n|<1$.  For the product
\eqref{eq:discrete-product} with \eqref{eq:ABdef}, 
\begin{equation}\label{eq:main-discrete}
 \big\|(\log|a|^2)^{1/2}\big\|_{L^{p'}(\T)}
 \le
 \big\|(\log A_n^2)^{1/2}\big\|_{\ell^p(\mathbb Z)}
\end{equation}
for all $1\le p\le2$.
For a sequence with a single nonzero coefficient equality holds for every
$p$.  For arbitrary finite sequences equality holds at $p=2$ by
\eqref{eq:Verblunsky}.
\end{theorem}

The continuous consequence of Theorem \ref{thm:discrete} solves the Muscalu--Tao--Thiele uniformity
problem.

\begin{theorem}\label{thm:continuous}
Let $1\le p\le2$ and $f\in L^1(\R)\cap L^p(\R)$.  Then the
Dirac scattering transform \eqref{eq:dirac-left} satisfies
\begin{equation}\label{eq:main-continuous}
 \big\|(\log|a_f|^2)^{1/2}\big\|_{L^{p'}(\R)}
 \le \norm{f}_{L^p(\R)}.
\end{equation}
At $p=2$ equality holds.
\end{theorem}

For $1\le p<2$ the space $L^p(\R)$ is not contained in $L^1(\R)$, and the
scattering transform of a general $f\in L^p(\R)$ is defined only through a
limiting procedure.  Write
\begin{equation}\label{eq:truncation}
 f_R:=f\,\mathbf 1_{[-R,R]},\qquad R>0.
\end{equation}
By H\"older's inequality,
\[
 \norm{f_R}_{L^1(\R)}
 \le (2R)^{1/p'}\norm{f_R}_{L^p(\R)}
 \le (2R)^{1/p'}\norm{f}_{L^p(\R)},
\]
so $f_R\in L^1(\R)\cap L^p(\R)$ and $a_{f_R}$ is defined by
\eqref{eq:dirac-left} in the classical sense.  Christ and Kiselev
\cite{CK1,CK2} showed that for every $f\in L^p(\R)$ with $1\le p<2$ the
scattering data of the truncated potentials converge pointwise almost
everywhere. In particular,
\begin{equation}\label{eq:CK-limit}
 a_f(\xi):=\lim_{R\to\infty}a_{f_R}(\xi)
\end{equation}
exists for almost every $\xi\in\R$.  This is the standard meaning of the nonlinear Fourier transform on
$L^p(\R)$; see also \cite{TT}.  The endpoint pointwise-convergence problem for
a version of the Dirac scattering transform was subsequently settled by
Poltoratski \cite{Poltoratski24}.  Theorem \ref{thm:continuous} extends to
the $L^p$ setting for $p<2$.

\begin{corollary}\label{cor:Lp}
Let $1\le p<2$ and let $f\in L^p(\R)$, with $a_f$ given by
\eqref{eq:CK-limit}.  Then
\begin{equation}\label{eq:main-Lp}
 \big\|(\log|a_f|^2)^{1/2}\big\|_{L^{p'}(\R)}
 \le\norm{f}_{L^p(\R)}.
\end{equation}
\end{corollary}

In particular, the constants $C_p$ of \eqref{eq:knownNLHY} may be taken equal
to $1$ for every $1<p<2$, which answers the Muscalu--Tao--Thiele uniformity
problem in the affirmative.

\subsection{About the proof}
The proof is based on the reproducing kernel $-\log(1-z\overline w)$ of the
Dirichlet space of functions vanishing at the origin.  Define
\begin{equation}\label{eq:Phi-intro}
 \Phi(w)=\left(w,\frac{w^2}{\sqrt2},\frac{w^3}{\sqrt3},\ldots\right)
 \in\ell^2(\N),\qquad w\in\D,
\end{equation}
where $\D:=\{z\in\mathbb C:|z|<1\}$ is the open unit disk.
Then
\begin{equation}\label{eq:Phi-energy-intro}
 \norm{\Phi(w)}_{\ell^2}^2=-\log(1-|w|^2).
\end{equation}
The crucial fact is that every disk automorphism acts on these feature
vectors by an affine complex-unitary map.  After a
simple gauge, the reflection coefficient of \eqref{eq:discrete-product}
obeys a Schur recursion.  Iterating the affine identity decomposes
$\Phi$ of the final reflection coefficient into Hilbert-valued increments
$v_j$.  Each increment has constant pointwise Hilbert norm, and distinct
increments are orthogonal in $L^2(\T;\ell^2)$ because all relative
Fourier frequencies are strictly one-sided.  Once the nonlinear background
is frozen, the associated synthesis operator is an ordinary linear map
with endpoint norms $1$ from $\ell^1$ to $L^\infty(\ell^2)$ and from
$\ell^2$ to $L^2(\ell^2)$.  A direct Hadamard three-lines argument then
proves Theorem \ref{thm:discrete}.

The continuous theorem is obtained by sampling a compactly supported
potential at mesh size $h$, choosing $F_n=h\overline{f(nh)}$, and evaluating
the discrete transform at $t=h\xi$.  The resulting ordered product is a
first-order approximation to the right Dirac flow.  We prove the required
uniform-on-compacts convergence directly, keep track of the exact scaling
of the discrete right-hand side, and then pass to the limit.

\medskip

Section \ref{sec:prelim} contains the preliminaries.  The discrete and
continuous theorems are proved in Sections \ref{sec:discrete-proof} and
\ref{sec:continuous-limit}, respectively.

\section{Preliminaries}\label{sec:prelim}

\subsection{The logarithmic kernel}
Let
\[
 \Hh:=\ell^2(\N)
 =\left\{x=(x_m)_{m\ge1}:\sum_{m\ge1}|x_m|^2<\infty\right\},
\]
with inner product linear in the first variable.  The map $\Phi$ is defined
by \eqref{eq:Phi-intro}.  The kernel below is the reproducing kernel of the
Dirichlet subspace $\mathcal D_0:=\{h\in\mathcal D:h(0)=0\}$; see,
for example, \cite[Chapters~1-2]{EFKMR14}.

\begin{lemma}\label{lem:kernel}
For $u,v\in\D$,
\begin{equation}\label{eq:kernel-identity}
 \ip{\Phi(u)}{\Phi(v)}
 =\sum_{m=1}^\infty\frac{(u\overline v)^m}{m}
 =-\log(1-u\overline v).
\end{equation}
Consequently,
\begin{equation}\label{eq:distance-identity}
 \norm{\Phi(u)-\Phi(v)}_{\Hh}^2
 =\log\frac{|1-u\overline v|^2}
 {(1-|u|^2)(1-|v|^2)}.
\end{equation}
\end{lemma}

\begin{proof}
Since $|u\overline v|<1$, the series in \eqref{eq:kernel-identity} converges
absolutely and equals the analytic power series for $-\log(1-z)$.  Finally,
\[
 \norm{\Phi(u)-\Phi(v)}^2
 =\norm{\Phi(u)}^2+\norm{\Phi(v)}^2
 -2\Rea\ip{\Phi(u)}{\Phi(v)},
\]
and substitution of \eqref{eq:kernel-identity} yields
\eqref{eq:distance-identity}.
\end{proof}

For $\alpha\in\D$ set
\begin{equation}\label{eq:phi-alpha}
 \varphi_\alpha(w):=\frac{\alpha+w}{1+\overline\alpha w}.
\end{equation}
This is a disk automorphism carrying $0$ to $\alpha$.
\begin{proposition}\label{prop:affine}
For every $\alpha\in\D$ there exists a unique complex-linear unitary
operator $U_\alpha:\Hh\to\Hh$ such that
\begin{equation}\label{eq:affine-action}
 \Phi(\varphi_\alpha(w))
 =\Phi(\alpha)+U_\alpha\Phi(w),\qquad w\in\D.
\end{equation}
The operator $U_\alpha$ depends only on $\alpha$.
\end{proposition}

\begin{proof}
For $u,v\in\D$ one has
\begin{align}
 1-\varphi_\alpha(u)\overline{\varphi_\alpha(v)}
 &=\frac{(1-|\alpha|^2)(1-u\overline v)}
 {(1+\overline\alpha u)(1+\alpha\overline v)},
 \label{eq:mobius-prod}\\
 1-\varphi_\alpha(u)\overline\alpha
 &=\frac{1-|\alpha|^2}{1+\overline\alpha u},
 \label{eq:mobius-alpha}\\
 1-\alpha\overline{\varphi_\alpha(v)}
 &=\frac{1-|\alpha|^2}{1+\alpha\overline v}.
 \label{eq:mobius-alpha2}
\end{align}
Using Lemma \ref{lem:kernel}, the desired preservation of complex Gram
matrices is equivalent to
\begin{align}
&-\log\bigl(1-\varphi_\alpha(u)
 \overline{\varphi_\alpha(v)}\bigr)
 +\log\bigl(1-\varphi_\alpha(u)\overline\alpha\bigr)
 \notag\\
&\qquad
 +\log\bigl(1-\alpha\overline{\varphi_\alpha(v)}\bigr)
 -\log(1-|\alpha|^2)
 =-\log(1-u\overline v).
 \label{eq:gram-log}
\end{align}
Equations \eqref{eq:mobius-prod}--\eqref{eq:mobius-alpha2} show that the
exponentials of the two sides agree.  Because logarithmic additivity may
otherwise carry a $2\pi i$ ambiguity, we spell out the branch argument.
Let $L(u,v)$ denote the difference between the two sides of
\eqref{eq:gram-log}, each logarithm being the analytic branch determined by
its convergent power series at the origin.  Then $L$ is continuous on the
connected set $\D\times\D$, while the exponential identities imply
$L(u,v)\in2\pi i\mathbb Z$.  Hence $L$ is constant.  Evaluating at
$(u,v)=(0,0)$ gives $L=0$.

Therefore
\begin{equation}\label{eq:gram-preserved}
 \ip{\Phi(\varphi_\alpha(u))-\Phi(\alpha)}
 {\Phi(\varphi_\alpha(v))-\Phi(\alpha)}
 =\ip{\Phi(u)}{\Phi(v)}.
\end{equation}
Define on the algebraic span of $\Phi(\D)$
\[
 U_\alpha\Bigl(\sum_{j=1}^Jc_j\Phi(w_j)\Bigr)
 :=\sum_{j=1}^Jc_j
 \bigl(\Phi(\varphi_\alpha(w_j))-\Phi(\alpha)\bigr).
\]
Identity \eqref{eq:gram-preserved} shows simultaneously that this definition
is independent of the chosen representation and that $U_\alpha$ preserves
the complex inner product.  Thus it extends uniquely to an isometry on the
closure of $\operatorname{span}\Phi(\D)$.

We now verify that this closed span is all of $\Hh$.  If
$x=(x_m)_{m\ge1}\in\Hh$ is orthogonal to every $\Phi(w)$, then
\[
 0=\ip{x}{\Phi(w)}
 =\sum_{m\ge1}x_m\frac{\overline w^{\,m}}{\sqrt m},
 \quad w\in\D.
\]
The right-hand side is an anti-holomorphic function of $w$ whose Taylor
coefficients are $x_m/\sqrt m$; all coefficients vanish, and hence $x=0$.
So the domain is dense.

To see that the range is dense, suppose $x$ is orthogonal to
$\Phi(\varphi_\alpha(w))-\Phi(\alpha)$ for every $w$.  Since
$\varphi_\alpha$ is onto $\D$, the function
\[
 z\longmapsto\ip{x}{\Phi(z)}
\]
is constant on $\D$.  Its value at $z=0$ is zero, because $\Phi(0)=0$.
Thus $x\perp\Phi(\D)$ and hence $x=0$.  The range of an isometry is closed,
so it is all of $\Hh$.  Therefore $U_\alpha$ is unitary.  Uniqueness follows
from density of $\operatorname{span}\Phi(\D)$.
\end{proof}

\begin{remark}\label{rem:coeff-U}
With the convention that $(U_\alpha)_{n,m}$ denotes the entry in row $n$
and column $m$, expansion of \eqref{eq:affine-action} in powers of $w$ gives
\begin{equation}\label{eq:Ucoeff}
 (U_\alpha)_{n,m}
 =\sqrt{\frac{m}{n}}\,[w^m]\,\varphi_\alpha(w)^n,
 \qquad m,n\ge1.
\end{equation}
 Proposition \ref{prop:affine} may also be
viewed as the centered kernel form of conformal invariance of the classical
Dirichlet seminorm (see e.g. \cite{EFKMR14}).
\end{remark}

For a unimodular $\zeta\in\C$ and $d\in\N$, define the diagonal unitary
\begin{equation}\label{eq:Dzd}
 D_{\zeta,d}(x_1,x_2,\ldots)
 =(\zeta^dx_1,\zeta^{2d}x_2,\zeta^{3d}x_3,\ldots).
\end{equation}
Then
\begin{equation}\label{eq:rotation-Phi}
 \Phi(\zeta^dw)=D_{\zeta,d}\Phi(w).
\end{equation}
Equations \eqref{eq:affine-action} and \eqref{eq:rotation-Phi} are the only
geometric facts required in the proof.

\medskip

Now we derive the Schur recursion in the ordering of
\eqref{eq:discrete-product}.  Although this computation is elementary, we
include it for completeness.  The recursion is consistent
with the standard Szeg\H{o}/Schur formalism (cf.~\cite{Simon1,TT,KOR22}).

If $F\equiv0$, then $a\equiv1$ and Theorem \ref{thm:discrete} is
immediate.  Henceforth assume that $F$ has nonempty support, and write
\[
 n_0<n_1<\cdots<n_N
\]
for its support.  For $0\le k\le N$ put
\begin{equation}\label{eq:Pk}
 P_k(t)
 :=\prod_{j=0}^{k}\!{}^{\longrightarrow}
 \begin{pmatrix}
  A_{n_j}&B_{n_j}\e^{2\pi in_jt}\\
  \overline{B_{n_j}}\e^{-2\pi in_jt}&A_{n_j}
 \end{pmatrix}
 =\begin{pmatrix}a_k&b_k\\\overline b_k&\overline a_k\end{pmatrix}.
\end{equation}
By right multiplication with the new factor,
\begin{align}
 a_k
 &=A_{n_k}a_{k-1}
 +\overline{B_{n_k}}\e^{-2\pi in_kt}b_{k-1},
 \label{eq:akrec}\\
 b_k
 &=B_{n_k}\e^{2\pi in_kt}a_{k-1}
 +A_{n_k}b_{k-1}.
 \label{eq:bkrec}
\end{align}
For the first factor one has
$a_0=A_{n_0}$ and $b_0=B_{n_0}\e^{2\pi in_0t}$.

Since $P_k(t)\in\SU$, $|a_k|^2-|b_k|^2=1$ and therefore
\[
 r_k(t):=\frac{b_k(t)}{a_k(t)}\in\D.
\]
Dividing \eqref{eq:bkrec} by \eqref{eq:akrec} and using $B_n=A_nF_n$
gives
\begin{equation}\label{eq:rrec-exact}
 r_k(t)
 =\frac{F_{n_k}\e^{2\pi in_kt}+r_{k-1}(t)}
 {1+\overline{F_{n_k}}\e^{-2\pi in_kt}r_{k-1}(t)}.
\end{equation}
Set
\begin{equation}\label{eq:s-gauge}
 s_k(t):=\e^{-2\pi in_kt}r_k(t).
\end{equation}
For $k\ge1$ write $d_k=n_k-n_{k-1}\ge1$.  Then
\begin{equation}\label{eq:Schur}
 s_0(t)=F_{n_0},\qquad
 s_k(t)=\varphi_{F_{n_k}}
 \bigl(\e^{-2\pi id_kt}s_{k-1}(t)\bigr).
\end{equation}
Thus the full nonlinear product is an iteration of disk automorphisms
separated by one-sided rotations.

The scattering energy is indeed the logarithmic disk energy of $s_k$:
\begin{equation}\label{eq:energy-sk}
 1-|s_k(t)|^2
 =1-\frac{|b_k(t)|^2}{|a_k(t)|^2}
 =|a_k(t)|^{-2},
\end{equation}
so
\begin{equation}\label{eq:energy-Phi-sk}
 \norm{\Phi(s_k(t))}_{\Hh}^2
 =-\log(1-|s_k(t)|^2)
 =\log|a_k(t)|^2.
\end{equation}

\section{Proof of Theorem \ref{thm:discrete}}\label{sec:discrete-proof}

Set, for $0\le k\le N$,
\begin{equation}\label{eq:alphacbeta}
 \alpha_k:=F_{n_k},\qquad
 c_k:=-\log(1-|\alpha_k|^2),\qquad
 \beta_k:=\Phi(\alpha_k).
\end{equation}
By Lemma \ref{lem:kernel},
\begin{equation}\label{eq:beta-norm}
 \norm{\beta_k}_{\Hh}^2=c_k=\log A_{n_k}^2.
\end{equation}
Since $n_k$ belongs to the support of $F$ we have $F_{n_k}\ne0$, and
$|F_{n_k}|<1$ by hypothesis; thus $0<|\alpha_k|<1$ and
\begin{equation}\label{eq:ck-positive}
 c_k>0,\qquad 0\le k\le N.
\end{equation}
Let $U_k:=U_{\alpha_k}$ be the unitary from Proposition
\ref{prop:affine} and, for $k\ge1$, put
\begin{equation}\label{eq:Dk}
 D_k(t):=D_{\e^{-2\pi it},d_k}.
\end{equation}
Then \eqref{eq:Schur}, \eqref{eq:affine-action} and
\eqref{eq:rotation-Phi} imply
\begin{equation}\label{eq:Phi-Schur}
 \Phi(s_k(t))=\beta_k+U_kD_k(t)\Phi(s_{k-1}(t)).
\end{equation}
Iterating \eqref{eq:Phi-Schur} gives
\begin{equation}\label{eq:cocycle-sum}
 \Phi(s_N(t))=\sum_{j=0}^Nv_j(t),
\end{equation}
where
\begin{equation}\label{eq:vj-def}
 v_N(t)=\beta_N
\end{equation}
and, for $0\le j<N$,
\begin{equation}\label{eq:vj-def2}
 v_j(t)
 =U_ND_N(t)U_{N-1}D_{N-1}(t)\cdots
 U_{j+1}D_{j+1}(t)\beta_j.
\end{equation}
Every factor in \eqref{eq:vj-def2} is unitary for each fixed $t$.
Consequently,
\begin{equation}\label{eq:v-constantnorm}
 \norm{v_j(t)}_{\Hh}=\norm{\beta_j}_{\Hh}=\sqrt{c_j},
 \quad t\in\T.
\end{equation}

\medskip

\begin{lemma}\label{lem:orthogonality}
For $j\ne k$,
\begin{equation}\label{eq:pair-orthogonal}
 \int_\T\ip{v_j(t)}{v_k(t)}\dd t=0.
\end{equation}
Hence, by \eqref{eq:ck-positive}, the functions
\begin{equation}\label{eq:ej-def}
 e_j(t):=\frac{v_j(t)}{\sqrt{c_j}},\qquad 0\le j\le N,
\end{equation}
form an orthonormal system in $L^2(\T;\Hh)$ and satisfy
$\norm{e_j(t)}_{\Hh}=1$ pointwise.
\end{lemma}

\begin{proof}
We use vector-valued Fourier coefficients in the Bochner sense.  For
$G\in L^2(\T;\Hh)$ define
\[
 \widehat G(\ell)=\int_\T G(t)\e^{-2\pi i\ell t}\dd t\in\Hh.
\]
Let
\begin{equation}\label{eq:Nminus}
 \mathcal N_-:=\{G\in L^2(\T;\Hh):\widehat G(\ell)=0
 \text{ for every }\ell\ge0\}.
\end{equation}
Since the Fourier coefficient maps are continuous,
$\mathcal N_-$ is a closed subspace.

We record two invariances.

\emph{First}, if $U$ is a constant unitary on $\Hh$ and
$G\in\mathcal N_-$, then
\[
 \widehat{UG}(\ell)=U\widehat G(\ell)=0,\quad \ell\ge0,
\]
so $UG\in\mathcal N_-$.  Note that the
operators $U_k$ depend on the coefficients $F$ but not on $t$.

\emph{Second}, for any positive integer $d$, the multiplication operator
$D_{\e^{-2\pi it},d}$ preserves $\mathcal N_-$.  Indeed, if
$G=(G_m)_{m\ge1}$, then the $m$th coordinate is multiplied by
$\e^{-2\pi imdt}$.  Thus its scalar spectrum is translated by $-md\le-1$.
Every strictly negative frequency remains strictly negative.

Since $\ip{v_k(t)}{v_j(t)}=\overline{\ip{v_j(t)}{v_k(t)}}$, it suffices to
treat $j<k$.  Now fix $j<k$.  In the pointwise inner product
$\ip{v_j(t)}{v_k(t)}$, cancel the common left prefix
\[
 U_ND_N(t)\cdots U_{k+1}D_{k+1}(t),
\]
which is unitary for each fixed $t$.  We obtain
\begin{equation}\label{eq:g-jk}
 \ip{g_{j,k}(t)}{\beta_k},
\end{equation}
where
\[
 g_{j,k}(t)
 =U_kD_k(t)U_{k-1}D_{k-1}(t)\cdots
 U_{j+1}D_{j+1}(t)\beta_j.
\]
The vector $D_{j+1}(t)\beta_j$ belongs to $\mathcal N_-$: in its $m$th
coordinate the only frequency is $-md_{j+1}\le-1$.  Applying successively
the two invariances just proved shows that $g_{j,k}\in\mathcal N_-$.  Hence
$\widehat g_{j,k}(0)=0$, and therefore
\[
 \int_\T\ip{g_{j,k}(t)}{\beta_k}\dd t
 =\ip{\widehat g_{j,k}(0)}{\beta_k}=0.
\]
This proves \eqref{eq:pair-orthogonal}.  The final assertion follows from
\eqref{eq:v-constantnorm}.
\end{proof}

We now interpolate a linear synthesis operator obtained by freezing the
nonlinear background.  We include the direct three-lines argument to avoid
any ambiguity about a vector-valued $L^\infty$ endpoint.  The scalar
three-lines theorem and this standard interpolation device can be found in \cite[Chapter V]{SteinWeiss71}, for example.

\begin{lemma}\label{lem:interpolation}
Let $(e_j)_{j\in J}$ be a finite orthonormal system in
$L^2(\T;\Hh)$ such that $\norm{e_j(t)}_{\Hh}\le1$ for every $j$ and a.e.
$t$.  Define
\[
 Tx(t)=\sum_{j\in J}x_je_j(t).
\]
Then, for $1\le p\le2$, we have 
\begin{equation}\label{eq:synthesis-HY}
 \norm{Tx}_{L^{p'}(\T;\Hh)}\le\norm{x}_{\ell^p(J)}.
\end{equation}
\end{lemma}

\begin{proof}
The endpoints are immediate.  Pointwise triangle inequality gives
\begin{equation}\label{eq:endpoint-1}
 \norm{Tx}_{L^\infty(\T;\Hh)}\le\norm{x}_{\ell^1},
\end{equation}
while orthonormality gives
\begin{equation}\label{eq:endpoint-2}
 \norm{Tx}_{L^2(\T;\Hh)}=\norm{x}_{\ell^2}.
\end{equation}

Fix $1<p<2$ and choose $0<\theta<1$ so that
\begin{equation}\label{eq:theta-rel}
 \frac1p=1-\frac\theta2,
 \qquad
 \frac{1}{p'}=\frac\theta2.
\end{equation}
By homogeneity it is enough to prove
\eqref{eq:synthesis-HY} for finite $x$ and against simple $\Hh$-valued
functions $G$, with
\begin{equation}\label{eq:normed-xG}
 \norm{x}_{\ell^p}=1,
 \qquad
 \norm{G}_{L^p(\T;\Hh)}=1.
\end{equation}
Write $\operatorname{sgn}x_j=x_j/|x_j|$ when $x_j\ne0$ and $0$ otherwise, and
$\operatorname{sgn}G(t)=G(t)/\norm{G(t)}$ when $G(t)\ne0$ and $0$ otherwise.
Let
\[
 S:=\{z\in\C:\ 0\le\Rea z\le1\}
\]
and, for $z\in S$, define
\begin{align}
 x_j(z)
 &:=\operatorname{sgn}x_j\,
 |x_j|^{p(1-z/2)},
 \label{eq:xz}\\
 G_z(t)
 &:=\operatorname{sgn}G(t)\,
 \norm{G(t)}^{p(1-\overline z/2)},
 \label{eq:Gz}
\end{align}
where $a^{w}:=\e^{w\log a}$ for $a>0$ and $0^{w}:=0$.  The latter convention is
consistent because the exponent has positive real part throughout $S$:
\begin{equation}\label{eq:exponent-range}
 \Rea\Bigl(p\Bigl(1-\frac z2\Bigr)\Bigr)
 =p\Bigl(1-\frac{\Rea z}{2}\Bigr)\in\Bigl[\frac p2,\ p\Bigr],
 \qquad z\in S.
\end{equation}

Since $\norm{G(t)}$ is real and nonnegative we have
$\overline{\norm{G(t)}^{\,p(1-\overline z/2)}}=\norm{G(t)}^{\,p(1-z/2)}$, and the
inner product is linear in its first variable and conjugate linear in its
second.  Hence the function
\begin{equation}\label{eq:Psi-strip}
 \Psi(z):=\int_\T\ip{T x(z)(t)}{G_z(t)}\dd t
 =\sum_{j\in J}x_j(z)\int_\T
 \norm{G(t)}^{\,p(1-z/2)}\,
 \ip{e_j(t)}{\operatorname{sgn}G(t)}\dd t
\end{equation}
is a finite sum of products of entire functions of $z$ with integrals of
$z$-analytic integrands.

We claim that $\Psi$ is analytic in the interior of $S$, continuous on $S$, and
\emph{bounded on all of} $S$.  Indeed, for $a\ge0$ and $\tfrac p2\le s\le p$ one
has $a^{s}\le\max(a^{p/2},a^{p})\le1+a^{p}$, so \eqref{eq:exponent-range} gives
the $z$-independent majorants
\begin{equation}\label{eq:strip-domination}
 |x_j(z)|\le1+|x_j|^{p},
 \qquad
 \norm{G_z(t)}_{\Hh}\le1+\norm{G(t)}_{\Hh}^{p},
 \qquad z\in S.
\end{equation}
The second majorant is integrable on $\T$ by \eqref{eq:normed-xG}, so
analyticity in the interior of $S$ and continuity on $S$ follow from
\eqref{eq:Psi-strip} by dominated convergence together with Morera's theorem,
the index set $J$ being finite.  Furthermore, using $\norm{e_j(t)}_{\Hh}\le1$,
the Cauchy--Schwarz inequality in $\Hh$, \eqref{eq:strip-domination},
\eqref{eq:normed-xG} and $|\T|=1$,
\begin{equation}\label{eq:Psi-bounded}
 |\Psi(z)|
 \le\Bigl(\sum_{j\in J}|x_j(z)|\Bigr)
 \int_\T\norm{G_z(t)}_{\Hh}\dd t
 \le\bigl(\#J+\norm{x}_{\ell^p}^{p}\bigr)
 \bigl(1+\norm{G}_{L^p(\T;\Hh)}^{p}\bigr)
 =2\bigl(\#J+1\bigr)
\end{equation}
for every $z\in S$.  The bound \eqref{eq:Psi-bounded} is crude; only its
independence of $z$ matters, and it is indeed the boundedness hypothesis
required by Hadamard's three-lines theorem.

If $z=iy$, then
\[
 \norm{x(iy)}_{\ell^1}
 =\sum_j|x_j|^p=1,
 \qquad
 \norm{G_{iy}}_{L^1}
 =\int_\T\norm{G(t)}^p\dd t=1.
\]
Using \eqref{eq:endpoint-1} and $L^\infty$--$L^1$ duality,
\begin{equation}\label{eq:Psi-left}
 |\Psi(iy)|\le1.
\end{equation}
If $z=1+iy$, then
\[
 \norm{x(1+iy)}_{\ell^2}^2=\sum_j|x_j|^p=1,
 \qquad
 \norm{G_{1+iy}}_{L^2}^2
 =\int_\T\norm{G(t)}^p\dd t=1.
\]
By \eqref{eq:endpoint-2} and the Cauchy--Schwarz inequality,
\begin{equation}\label{eq:Psi-right}
 |\Psi(1+iy)|\le1.
\end{equation}
Hadamard's three-lines theorem therefore yields $|\Psi(\theta)|\le1$.
By \eqref{eq:theta-rel}, $x(\theta)=x$ and $G_\theta=G$.  Thus
\begin{equation}\label{eq:duality-bound}
 \left|\int_\T\ip{Tx(t)}{G(t)}\dd t\right|\le1
\end{equation}
for all simple $G$ with $L^p$ norm one.  Since $1<p<\infty$ and $\Hh$ is a
separable Hilbert space, the usual Bochner-space duality
$(L^{p'}(\T;\Hh))^*=L^p(\T;\Hh)$ applies (see
\cite[Chapter IV]{DiestelUhl77}).  Taking the supremum in
\eqref{eq:duality-bound} gives \eqref{eq:synthesis-HY}.  The endpoint cases
are \eqref{eq:endpoint-1} and \eqref{eq:endpoint-2}.
\end{proof}

\begin{proof}[Proof of Theorem \ref{thm:discrete}]
 With the orthonormal
system \eqref{eq:ej-def}, define
\[
 T_Fx(t):=\sum_jx_je_j(t).
\]
The subscript emphasizes that $e_j$ depends on the original nonlinear
background $F$.  Once $F$ is fixed, however, $T_F$ is an ordinary linear
operator.  Lemma \ref{lem:interpolation} yields
\begin{equation}\label{eq:TFinterpolated}
 \norm{T_Fx}_{L^{p'}(\T;\Hh)}\le\norm{x}_{\ell^p}.
\end{equation}
Choose $x_j=\sqrt{c_j}$.  By \eqref{eq:cocycle-sum},
\[
 T_Fx=\sum_jv_j=\Phi(s_N).
\]
By \eqref{eq:energy-Phi-sk},
\[
 \norm{\Phi(s_N(t))}_{\Hh}
 =\bigl(\log|a(t)|^2\bigr)^{1/2}.
\]
The right-hand side of \eqref{eq:TFinterpolated} is
\[
 \left(\sum_jc_j^{p/2}\right)^{1/p}
 =\left(\sum_n[-\log(1-|F_n|^2)]^{p/2}\right)^{1/p}.
\]
This proves \eqref{eq:main-discrete}.  If only one coefficient is nonzero,
then $\log|a(t)|^2$ is constant and both sides coincide.  At $p=2$ the
proof reduces to Pythagoras' theorem for the orthogonal decomposition
\eqref{eq:cocycle-sum}, which is \eqref{eq:Verblunsky}.
\end{proof}

\begin{remark}\label{rem:no-nonlinear-interp}
No nonlinear interpolation principle has been used.  The dependence of the
orthonormal system $(e_j)$ on $F$ is frozen before Lemma
\ref{lem:interpolation} is applied.  Only after the linear estimate is
proved do we insert the particular coefficient vector
$x_j=\sqrt{-\log(1-|F_{n_j}|^2)}$.
\end{remark}

\section{Proof of Theorem \ref{thm:continuous}}\label{sec:continuous-limit}

We now transfer Theorem \ref{thm:discrete} to the continuous Dirac
scattering problem.  

\subsection{Left and right Dirac evolutions}
For $g\in L^1(\R)$ and $\xi\in\R$, define
\begin{equation}\label{eq:Wdef}
 W_g(x,\xi)
 :=\begin{pmatrix}
 0&\overline{g(x)}\e^{2\pi ix\xi}\\
 g(x)\e^{-2\pi ix\xi}&0
 \end{pmatrix}.
\end{equation}
We denote by $P_g^{\mathrm R}(x,\xi)$ the unique solution of the Volterra
integral equation
\begin{equation}\label{eq:right-Volterra}
 P_g^{\mathrm R}(x,\xi)
 =
 I+\int_{-\infty}^{x}
 P_g^{\mathrm R}(s,\xi)W_g(s,\xi)\,\dd s.
\end{equation}
Equivalently, $P_g^{\mathrm R}$ is the absolutely continuous solution of
\begin{equation}\label{eq:right-ODE}
 \partial_xP_g^{\mathrm R}(x,\xi)
 =
 P_g^{\mathrm R}(x,\xi)W_g(x,\xi),
 \qquad
 P_g^{\mathrm R}(-\infty,\xi)=I.
\end{equation}
Existence and uniqueness follow from the standard Volterra iteration.
Moreover, since
\[
 \norm{W_g(x,\xi)}=|g(x)|,
\]
Gronwall's inequality gives, uniformly in $x,\xi\in\R$,
\begin{equation}\label{eq:right-basic-bound}
 \norm{P_g^{\mathrm R}(x,\xi)}
 \le
 \exp\!\left(\int_{-\infty}^{x}|g(s)|\,\dd s\right)
 \le \e^{\norm{g}_{L^1}}.
\end{equation}
In particular,
\[
 \int_{\R}
 \norm{P_g^{\mathrm R}(s,\xi)W_g(s,\xi)}\,\dd s
 \le
 \e^{\norm{g}_{L^1}}\norm{g}_{L^1},
\]
and therefore the limit
\begin{equation}\label{eq:right-terminal-limit}
 P_g^{\mathrm R}(+\infty,\xi)
 :=
 \lim_{x\to+\infty}P_g^{\mathrm R}(x,\xi)
\end{equation}
exists in matrix norm for every $\xi\in\R$.  More precisely,
\begin{equation}\label{eq:right-tail-bound}
 \norm{
 P_g^{\mathrm R}(+\infty,\xi)-P_g^{\mathrm R}(x,\xi)}
 \le
 \e^{\norm{g}_{L^1}}
 \int_x^\infty |g(s)|\,\dd s,
\end{equation}
uniformly in $\xi$.

The flow preserves the group $\SU$.  Indeed, with $J:=\operatorname{diag}(1,-1)$
one checks directly from \eqref{eq:Wdef} that
$W_g(x,\xi)J+JW_g(x,\xi)^*=0$ and $\operatorname{tr}W_g(x,\xi)=0$, whence
\[
 \partial_x\bigl(P_g^{\mathrm R}J(P_g^{\mathrm R})^*\bigr)
 =P_g^{\mathrm R}\bigl(W_gJ+JW_g^*\bigr)(P_g^{\mathrm R})^*=0,
 \qquad
 \partial_x\det P_g^{\mathrm R}
 =\bigl(\det P_g^{\mathrm R}\bigr)\operatorname{tr}W_g=0.
\]
Since $P_g^{\mathrm R}(-\infty,\xi)=I$, this gives
$P_g^{\mathrm R}J(P_g^{\mathrm R})^*\equiv J$ and $\det P_g^{\mathrm R}\equiv1$, that is,
$P_g^{\mathrm R}(x,\xi)\in\SU$ for all $x$ and $\xi$; letting $x\to+\infty$, the
same holds for the terminal matrix \eqref{eq:right-terminal-limit}.

We may therefore write
\begin{equation}\label{eq:right-terminal}
 P_g^{\mathrm R}(+\infty,\xi)
 =
 \begin{pmatrix}
 a_g^{\mathrm R}(\xi)&b_g^{\mathrm R}(\xi)\\
 \overline{b_g^{\mathrm R}(\xi)}&
 \overline{a_g^{\mathrm R}(\xi)}
 \end{pmatrix},
\end{equation}
and, since $\det P_g^{\mathrm R}(+\infty,\xi)=1$,
\begin{equation}\label{eq:right-SU-relation}
 |a_g^{\mathrm R}(\xi)|^2-|b_g^{\mathrm R}(\xi)|^2=1.
\end{equation}
Thus the right transmission coefficient $a_g^{\mathrm R}$ is well defined
for every $g\in L^1(\R)$
(see the scattering discussion of
\cite{TT}).
We write $P_f^{\mathrm L}(x,\xi)$ for the corresponding solution of the left
system, that is, the matrix solution of
\[
 \partial_xP_f^{\mathrm L}(x,\xi)
 =\begin{pmatrix}
 0&\overline{f(x)}\e^{2\pi ix\xi}\\
 f(x)\e^{-2\pi ix\xi}&0
 \end{pmatrix}P_f^{\mathrm L}(x,\xi),
 \qquad
 P_f^{\mathrm L}(-\infty,\xi)=I ;
\]
its first column is the solution $(a,b)$ of \eqref{eq:dirac-left}, so that
$a_f$ is the $(1,1)$ entry of $P_f^{\mathrm L}(+\infty,\xi)$.

\begin{lemma}\label{lem:left-right}
For every $g\in L^1(\R)$,
\begin{equation}\label{eq:right-left-exact}
 a_g^{\mathrm R}(\xi)=a_{\overline g}(-\xi),
\end{equation}
where $a_{\overline g}$ is defined by the left system
\eqref{eq:dirac-left}.
\end{lemma}

\begin{proof}
Transpose \eqref{eq:right-ODE}.  Then
\[
 \partial_x(P_g^{\mathrm R})^T
 =W_g(x,\xi)^T(P_g^{\mathrm R})^T,
\]
with
\[
 W_g(x,\xi)^T
 =\begin{pmatrix}
 0&g(x)\e^{-2\pi ix\xi}\\
 \overline{g(x)}\e^{2\pi ix\xi}&0
 \end{pmatrix}.
\]
This is the left coefficient matrix in \eqref{eq:dirac-left} for
the potential $\overline g$ and frequency $-\xi$.  Both systems have the
same initial matrix $I$.  Uniqueness of the Volterra equation implies
$(P_g^{\mathrm R})^T=P_{\overline g}^{\mathrm L}(\cdot,-\xi)$, and taking the
$(1,1)$ entry gives \eqref{eq:right-left-exact}.
\end{proof}

\subsection{Approximation by discrete \texorpdfstring{$\SU$}{SU(1,1)} factors}
We first assume $g\in C_c(\R)$.  If $g\equiv0$, every assertion below is
trivial, so assume $g\not\equiv0$.  Choose $R>0$ so that
$\supp g\subset(-R,R)$, and set
\[
 \rho:=\operatorname{dist}\bigl(\supp g,\R\setminus(-R,R)\bigr)>0,
 \qquad
 h_0:=\min\left\{1,\frac{\rho}{2},\frac{1}{2\norm g_\infty}\right\}.
\]
For $0<h\le h_0$, define
\[
 n_-(h):=\left\lceil-\frac Rh\right\rceil,
 \qquad
 n_+(h):=\left\lfloor\frac Rh\right\rfloor,
 \qquad x_n:=nh,
\]
and set
\begin{equation}\label{eq:Fnh}
 F_n^{(h)}:=
 \begin{cases}
  h\overline{g(nh)},&n_-(h)\le n\le n_+(h),\\
  0,&\text{otherwise}.
 \end{cases}
\end{equation}
The restriction $h\le h_0$ entails
$|F_n^{(h)}|\le h\norm g_\infty\le\tfrac12<1$.  It also ensures that
$x_{n_-(h)}$ lies strictly to the left of $\supp g$, whereas
$x_{n_+(h)+1}$ lies strictly to its right.  Let $P_h(t)$ be the finite
discrete product \eqref{eq:discrete-product} built from $F^{(h)}$, and let
$a_h(t)$ denote its $(1,1)$ entry.

We shall use the operator norm on $2\times2$ matrices.  The following
finite-product estimate is elementary but convenient.

\begin{lemma}\label{lem:product-estimate}
Let $X_1,\ldots,X_N$ and $Y_1,\ldots,Y_N$ be matrices satisfying
\[
 \norm{X_j},\norm{Y_j}\le1+Mh,
 \qquad
 \norm{X_j-Y_j}\le Ch^2,
 \qquad Nh\le L.
\]
Then
\begin{equation}\label{eq:telescoping-bound}
 \left\|\prod_{j=1}^NX_j-\prod_{j=1}^NY_j\right\|
 \le CLh\,\e^{2ML}.
\end{equation}
\end{lemma}

\begin{proof}
We have 
\[
 \prod_{j=1}^NX_j-\prod_{j=1}^NY_j
 =\sum_{k=1}^N
 \left(\prod_{j<k}X_j\right)(X_k-Y_k)
 \left(\prod_{j>k}Y_j\right).
\]
The norm of every partial product is at most
$(1+Mh)^N\le\e^{ML}$.  Summing $N$ terms of size at most
$Ch^2\e^{2ML}$ and using $Nh\le L$ proves
\eqref{eq:telescoping-bound}.
\end{proof}

\begin{proposition}\label{prop:Euler}
Let $g\in C_c(\R)$.  For every compact interval $K\subset\R$,
\begin{equation}\label{eq:matrix-convergence}
 \sup_{\xi\in K}
 \norm{P_h(h\xi)-P_g^{\mathrm R}(+\infty,\xi)}\longrightarrow0,
 \quad h\downarrow0.
\end{equation}
In particular,
\begin{equation}\label{eq:a-convergence}
 \sup_{\xi\in K}|a_h(h\xi)-a_g^{\mathrm R}(\xi)|\longrightarrow0,  \quad h\downarrow0.
\end{equation}
\end{proposition}

\begin{proof}
For $n_-(h)\le n\le n_+(h)$, the $n$th discrete factor evaluated at
$t=h\xi$ is
\begin{align}
 M_{n,h}(\xi)
 &:=\frac1{\sqrt{1-h^2|g(nh)|^2}}
 \begin{pmatrix}
  1&h\overline{g(nh)}\e^{2\pi inh\xi}\\
  hg(nh)\e^{-2\pi inh\xi}&1
 \end{pmatrix}.
 \label{eq:Mn}
\end{align}
Put $A_{n,h}:=(1-h^2|g(nh)|^2)^{-1/2}$.  Comparing \eqref{eq:Mn}
with \eqref{eq:Wdef}, we have the exact factorization
\begin{equation}\label{eq:Mn-factored}
 M_{n,h}(\xi)=A_{n,h}\bigl(I+hW_g(nh,\xi)\bigr).
\end{equation}
Thus the entire discrepancy between $M_{n,h}$ and the Euler factor
$I+hW_g(nh,\xi)$ is carried by the scalar $A_{n,h}$.

For $0<h\le h_0$ we have
$0\le h^2|g(nh)|^2\le\tfrac14$.  Since
$1\le(1-y)^{-1/2}\le1+y$ for $0\le y\le\tfrac14$,
\begin{equation}\label{eq:A-close-to-1}
 0\le A_{n,h}-1\le h^2\norm g_\infty^2.
\end{equation}
Moreover, $W_g(x,\xi)$ is Hermitian with eigenvalues $\pm|g(x)|$, and hence
\[
 \norm{I+hW_g(nh,\xi)}=1+h|g(nh)|\le\tfrac32.
\]
It follows from \eqref{eq:Mn-factored} and \eqref{eq:A-close-to-1} that
\begin{equation}\label{eq:factor-expansion}
 M_{n,h}(\xi)=I+hW_g(nh,\xi)+R_{n,h}(\xi),
 \qquad
 \norm{R_{n,h}(\xi)}\le\tfrac32\norm g_\infty^2h^2.
\end{equation}
The same estimates, together with $h\le1$, also give a constant
$M_g<\infty$, depending only on $\norm g_\infty$, such that
\begin{equation}\label{eq:factor-norm-bounds}
 \norm{M_{n,h}(\xi)},\ \norm{I+hW_g(nh,\xi)}\le1+M_gh,
 \qquad
 \norm{R_{n,h}(\xi)}\le\tfrac32\norm g_\infty^2h^2.
\end{equation}
These bounds are uniform in $n$, $\xi$, and $h\le h_0$.

Let
\begin{equation}\label{eq:Qh}
 Q_h(\xi):=
 \prod_{n=n_-(h)}^{n_+(h)}\!{}^{\longrightarrow}
 \bigl(I+hW_g(nh,\xi)\bigr).
\end{equation}
Inserting identity factors at those indices where $g(nh)=0$ shows that the
corresponding product of the matrices $M_{n,h}(\xi)$ is $P_h(h\xi)$.
The number $N_h:=n_+(h)-n_-(h)+1$ of factors satisfies
$N_hh\le2R+h\le2R+1$.  Lemma \ref{lem:product-estimate},
\eqref{eq:factor-expansion}, and \eqref{eq:factor-norm-bounds} therefore give
\begin{equation}\label{eq:P-Q-error}
 \sup_{\xi\in\R}\norm{P_h(h\xi)-Q_h(\xi)}\le C_{R,g}h.
\end{equation}

It remains to compare $Q_h$ with the solution of \eqref{eq:right-ODE}.
Write $Y(x,\xi):=P_g^{\mathrm R}(x,\xi)$, and define the partial Euler products
by
\[
 Z_{n_-(h)}(\xi):=I,
 \qquad
 Z_{n+1}(\xi):=Z_n(\xi)\bigl(I+hW_g(x_n,\xi)\bigr),
 \quad n_-(h)\le n\le n_+(h).
\]
Then $Z_{n_+(h)+1}=Q_h$.  By the choice of $h_0$ and the support margin
$\rho$, we also have
\begin{equation}\label{eq:grid-terminal-identification}
 Y(x_{n_-(h)},\xi)=I,
 \qquad
 Y(x_{n_+(h)+1},\xi)=P_g^{\mathrm R}(+\infty,\xi).
\end{equation}
For $n_-(h)\le n\le n_+(h)$, the Volterra equation gives
\begin{equation}\label{eq:Volterra-Y}
 Y(x_{n+1},\xi)
 =Y(x_n,\xi)+\int_{x_n}^{x_{n+1}}Y(s,\xi)W_g(s,\xi)\dd s.
\end{equation}
Subtracting $Y(x_n,\xi)(I+hW_g(x_n,\xi))$ and using
$\norm{Y(s,\xi)}\le\e^{\norm g_1}=:\Gamma_g$, we obtain, uniformly for
$\xi$ in a fixed compact interval $K$,
\begin{align}
&\norm{Y(x_{n+1},\xi)
 -Y(x_n,\xi)(I+hW_g(x_n,\xi))}
 \notag\\
&\quad\le
 \int_{x_n}^{x_{n+1}}
 \norm{Y(s,\xi)-Y(x_n,\xi)}\,\norm{W_g(s,\xi)}\dd s
 \notag\\
&\qquad+
 \Gamma_g\int_{x_n}^{x_{n+1}}
 \norm{W_g(s,\xi)-W_g(x_n,\xi)}\dd s.
 \label{eq:local-Euler}
\end{align}
The first term is at most $C_gh^2$.  Let
\[
 \omega_g(h):=\sup_{|x-y|\le h}|g(x)-g(y)|,
\]
which tends to zero because $g$ is uniformly continuous.  If
$K\subset[-M,M]$, then the elementary bound
$|\e^{2\pi is\xi}-\e^{2\pi ix_n\xi}|\le2\pi M|s-x_n|$ yields
\begin{equation}\label{eq:W-modulus}
 \sup_{\substack{|s-x_n|\le h\\ \xi\in K}}
 \norm{W_g(s,\xi)-W_g(x_n,\xi)}
 \le \omega_g(h)+2\pi Mh\norm g_\infty.
\end{equation}
Consequently, the local truncation error in \eqref{eq:local-Euler} is bounded
by
\begin{equation}\label{eq:local-error-final}
 C_{K,g}h\bigl(\omega_g(h)+h\bigr).
\end{equation}

Set $E_n(\xi):=Z_n(\xi)-Y(x_n,\xi)$.  Equations
\eqref{eq:local-error-final} and \eqref{eq:factor-norm-bounds} imply
\[
 \norm{E_{n+1}(\xi)}
 \le(1+h\norm g_\infty)\norm{E_n(\xi)}
 +C_{K,g}h\bigl(\omega_g(h)+h\bigr).
\]
The initial error is zero by \eqref{eq:grid-terminal-identification}, and
$N_hh\le2R+1$.  A discrete Gronwall recursion therefore gives
\begin{equation}\label{eq:Euler-global}
 \sup_{\xi\in K}
 \norm{Q_h(\xi)-P_g^{\mathrm R}(+\infty,\xi)}
 \le C_{R,K,g}\bigl(\omega_g(h)+h\bigr)\longrightarrow0.
\end{equation}
Combining \eqref{eq:P-Q-error} and \eqref{eq:Euler-global} proves
\eqref{eq:matrix-convergence}, and hence \eqref{eq:a-convergence}.
\end{proof}

\subsection{Passage to the limit}
We first prove Theorem \ref{thm:continuous} for $g\in C_c(\R)$ in the right
normalization.  Let
\begin{equation}\label{eq:Hh}
 H_h(t):=\bigl(\log|a_h(t)|^2\bigr)^{1/2}.
\end{equation}
For $1<p<2$ and $q=p'$, Theorem \ref{thm:discrete} gives
\begin{equation}\label{eq:discrete-h}
 \norm{H_h}_{L^q(\T)}
 \le
 \left(\sum_n[-\log(1-h^2|g(nh)|^2)]^{p/2}\right)^{1/p}.
\end{equation}
Represent $\T$ by $[-1/2,1/2]$ and set
$I_h=[-1/(2h),1/(2h)]$.  The change of variables $t=h\xi$ gives
\begin{equation}\label{eq:scaled-Hh}
 \norm{H_h(h\cdot)}_{L^q(I_h)}
 \le
 h^{-1/q}
 \left(\sum_n[-\log(1-h^2|g(nh)|^2)]^{p/2}\right)^{1/p}.
\end{equation}
We now compute the limit of the right-hand side.  Uniformly in the finitely
many relevant $n$,
\begin{equation}\label{eq:log-expand}
 -\log(1-h^2|g(nh)|^2)
 =h^2|g(nh)|^2\bigl(1+O(h^2\norm g_\infty^2)\bigr).
\end{equation}
Therefore
\begin{align}
&h^{-p/q}
 \sum_n[-\log(1-h^2|g(nh)|^2)]^{p/2}
 \notag\\
&\quad=
 h^{-p/q}h^p
 \sum_n|g(nh)|^p
 \bigl(1+O(h^2\norm g_\infty^2)\bigr)
 \notag\\
&\quad=
 h\sum_n|g(nh)|^p
 +O(h^2\norm g_\infty^2)
 \left(h\sum_n|g(nh)|^p\right),
 \label{eq:RHS-limit}
\end{align}
where we used $p/q=p-1$.  Since $g$ is continuous and compactly supported,
the Riemann sums converge, so
\begin{equation}\label{eq:RHS-conv}
 h^{-1/q}
 \left(\sum_n[-\log(1-h^2|g(nh)|^2)]^{p/2}\right)^{1/p}
 \longrightarrow\norm g_{L^p(\R)}.
\end{equation}

Let
\[
 H_g^{\mathrm R}(\xi)
 :=\bigl(\log|a_g^{\mathrm R}(\xi)|^2\bigr)^{1/2}.
\]
For any fixed $M>0$, Proposition \ref{prop:Euler} gives uniform convergence
of $a_h(h\xi)$ to $a_g^{\mathrm R}(\xi)$ on $[-M,M]$.  The limiting function
is bounded there, and uniform convergence places all $a_h(h\xi)$, for small
$h$, in a common compact subset of $\{z:|z|\ge1\}$.  Since every relevant
matrix belongs to $\SU$ and the map
$z\mapsto(\log|z|^2)^{1/2}$ is uniformly continuous on that compact set, we
have uniform convergence
\begin{equation}\label{eq:Hlocaluniform}
 H_h(h\cdot)\longrightarrow H_g^{\mathrm R}
 \quad\text{on }[-M,M].
\end{equation}
For sufficiently small $h$, $[-M,M]\subset I_h$.  Hence
\eqref{eq:scaled-Hh}, \eqref{eq:RHS-conv} and
\eqref{eq:Hlocaluniform} imply
\begin{equation}\label{eq:local-cont-est}
 \norm{H_g^{\mathrm R}}_{L^q([-M,M])}\le\norm g_p.
\end{equation}
Letting $M\to\infty$ and using monotone convergence gives
\begin{equation}\label{eq:right-cont-p}
 \norm{H_g^{\mathrm R}}_{L^q(\R)}\le\norm g_p,
 \quad1<p<2.
\end{equation}

For $p=1$, the discrete theorem gives
\[
 \norm{H_h}_{L^\infty(\T)}
 \le\sum_n[-\log(1-h^2|g(nh)|^2)]^{1/2}.
\]
After the same scaling, the right-hand side converges to
$\int|g|$.  From the local uniform convergence we obtain for every $M$
\[
 \norm{H_g^{\mathrm R}}_{L^\infty([-M,M])}
 \le\liminf_{h\downarrow0}\norm{H_h(h\cdot)}_{L^\infty(I_h)}
 \le\norm g_1.
\]
Letting $M\to\infty$ yields the $p=1$ endpoint.  For $p=2$ we use the
nonlinear Plancherel identity \eqref{eq:nonlin-plancherel} (see
\cite{TT}).

By Lemma \ref{lem:left-right}, complex conjugation of the potential and
reflection $\xi\mapsto-\xi$ identify the right and left transmission
coefficients without changing any $L^p$ norm.  Thus Theorem
\ref{thm:continuous} holds for $C_c(\R)$.

\subsection{Extension from \texorpdfstring{$C_c$ to $L^1\cap L^p$}{compact support to L1 intersection Lp}}
Here we record the $L^1$ stability estimate needed for approximation.

\begin{lemma}\label{lem:L1stability}
Let $g,\widetilde g\in L^1(\R)$. Then
\begin{equation}\label{eq:L1stability}
 \sup_{x,\xi}
 \norm{P_g^{\mathrm R}(x,\xi)-P_{\widetilde g}^{\mathrm R}(x,\xi)}
 \le
 \e^{\norm{g}_{L^1}+\norm{\widetilde g}_{L^1}}
 \norm{g-\widetilde g}_{L^1}.
\end{equation}
Consequently,
\begin{equation}\label{eq:L1stability-terminal}
 \sup_{\xi\in\R}
 \norm{
 P_g^{\mathrm R}(+\infty,\xi)
 -
 P_{\widetilde g}^{\mathrm R}(+\infty,\xi)}
 \le
 \e^{\norm{g}_{L^1}+\norm{\widetilde g}_{L^1}}
 \norm{g-\widetilde g}_{L^1}.
\end{equation}
\end{lemma}

\begin{proof}
Both matrices solve the Volterra equation \eqref{eq:right-Volterra}, and by
\eqref{eq:right-basic-bound},
\begin{equation}\label{eq:Pg-bound}
 \sup_{x,\xi}\norm{P_g^{\mathrm R}(x,\xi)}\le\e^{\norm{g}_{L^1}},
 \qquad
 \sup_{x,\xi}\norm{P_{\widetilde g}^{\mathrm R}(x,\xi)}
 \le\e^{\norm{\widetilde g}_{L^1}}.
\end{equation}
Subtracting the two equations and adding and subtracting
$P_g^{\mathrm R}W_{\widetilde g}$ gives, for all $x$ and $\xi$,
\begin{align*}
 \norm{P_g^{\mathrm R}(x,\xi)-P_{\widetilde g}^{\mathrm R}(x,\xi)}
 &\le
 \int_{-\infty}^x
 \norm{P_g^{\mathrm R}(s,\xi)-P_{\widetilde g}^{\mathrm R}(s,\xi)}\,
 |\widetilde g(s)|\dd s\\
 &\quad+
 \e^{\norm{g}_{L^1}}\int_{-\infty}^x|g(s)-\widetilde g(s)|\dd s.
\end{align*}
A second application of Gronwall's inequality yields \eqref{eq:L1stability}.
Since both terminal matrices exist as limits \eqref{eq:right-terminal-limit},
letting $x\to+\infty$ in \eqref{eq:L1stability} gives
\eqref{eq:L1stability-terminal}.
\end{proof}

\begin{proof}[Completion of the proof of Theorem \ref{thm:continuous}]
Let $f\in L^1(\R)\cap L^p(\R)$, $1\le p<2$.  Choose
$f_m\in C_c(\R)$ such that
\begin{equation}\label{eq:approx-f}
 \norm{f_m-f}_1+\norm{f_m-f}_p\longrightarrow0.
\end{equation}
By \eqref{eq:L1stability-terminal}, the terminal matrices
$P^{\mathrm R}_{f_m}(+\infty,\cdot)$ converge to
$P^{\mathrm R}_{f}(+\infty,\cdot)$ uniformly in $\xi$.  Lemma
\ref{lem:left-right} gives the same conclusion for the left transmission
coefficients.  Moreover, \eqref{eq:right-basic-bound} and
\eqref{eq:approx-f} place $a_{f_m}(\xi)$ and $a_f(\xi)$ in a common compact
annulus $\{z:1\le|z|\le C_f\}$.  Hence
\begin{equation}\label{eq:H-uniform-approx}
 \norm{H_{f_m}-H_f}_{L^\infty(\R)}\longrightarrow0.
\end{equation}

If $p=1$, the compactly supported case and
\eqref{eq:H-uniform-approx} give
\[
 \norm{H_f}_{L^\infty}
 =\lim_{m\to\infty}\norm{H_{f_m}}_{L^\infty}
 \le\lim_{m\to\infty}\norm{f_m}_{L^1}
 =\norm f_{L^1}.
\]
If $1<p<2$, then pointwise convergence, Fatou's lemma, and the compactly
supported estimate yield
\[
 \norm{H_f}_{L^{p'}}
 \le\liminf_{m\to\infty}\norm{H_{f_m}}_{L^{p'}}
 \le\lim_{m\to\infty}\norm{f_m}_{L^p}
 =\norm f_{L^p}.
\]
For $p=2$, we use the nonlinear Plancherel identity
\eqref{eq:nonlin-plancherel}. This proves the theorem.
\end{proof}

\subsection{Passage from \texorpdfstring{$L^1\cap L^p$ to $L^p$}{L1 cap Lp to Lp}}

\begin{proof}[Proof of Corollary \ref{cor:Lp}]
If $p=1$ then $L^1(\R)\cap L^p(\R)=L^p(\R)$ and the assertion is precisely
Theorem \ref{thm:continuous}.  Assume therefore $1<p<2$ and set
$q=p'\in(2,\infty)$.

Let $f\in L^p(\R)$ and let $f_R$ be the truncation \eqref{eq:truncation}.
Each $f_R$ belongs to $L^1(\R)\cap L^p(\R)$, so Theorem \ref{thm:continuous}
applies to it.  Since $|f_R|\le|f|$ pointwise, we obtain the bound
\begin{equation}\label{eq:trunc-bound}
 \int_\R\bigl(\log|a_{f_R}(\xi)|^2\bigr)^{q/2}\dd\xi
 \;\le\;\norm{f_R}_{L^p(\R)}^{\,q}
 \;\le\;\norm{f}_{L^p(\R)}^{\,q},
\end{equation}
uniformly in $R>0$.

Fix any sequence $R_m\to\infty$.  By \eqref{eq:CK-limit},
$a_{f_{R_m}}(\xi)\to a_f(\xi)$ for almost every $\xi\in\R$.  Each
$a_{f_{R_m}}$ satisfies $|a_{f_{R_m}}|\ge1$ by \eqref{eq:SU-conservation},
so $|a_f|\ge1$ almost everywhere as well.  Pointwise continuity of
$z\mapsto(\log|z|^2)^{1/2}$ on $\{|z|\ge1\}$ then gives
\begin{equation}\label{eq:H-ae}
 H_{f_{R_m}}(\xi)\longrightarrow H_f(\xi)
 \qquad\text{for a.e. }\xi\in\R,
\end{equation}
with $H$ as in \eqref{eq:Hf-def}.  The functions $H_{f_{R_m}}$ are
nonnegative and measurable, so Fatou's lemma applied along the sequence
$(R_m)$, together with \eqref{eq:H-ae} and \eqref{eq:trunc-bound}, gives
\[
 \int_\R H_f(\xi)^q\dd\xi
 \le\liminf_{m\to\infty}\int_\R H_{f_{R_m}}(\xi)^q\dd\xi
 \le\norm{f}_{L^p(\R)}^{\,q}.
\]
Taking $q$-th roots yields \eqref{eq:main-Lp}.
\end{proof}

\begin{remark}
Only the uniformity of the constant in Theorem \ref{thm:continuous} is used
here: the bound \eqref{eq:trunc-bound} is independent of $R$, which is 
what fails for the Christ--Kiselev constants $C_p$ as $p\uparrow2$.
\end{remark}

\begin{remark}
The expansion of the nonlinear reflection coefficient into multilinear
oscillatory integrals is never estimated.  Instead, all powers are packaged
into the single Hilbert vector $\Phi(s)$.  The nonlinear Möbius evolution is
an affine unitary transformation of this vector, and the cancellation
needed at $p=2$ becomes ordinary Hilbert-space orthogonality.  This is why
the endpoint obstruction of \cite{MTTcounter}, which concerns separate
multilinear pieces, is bypassed.
\end{remark}

\begin{remark}
For $1<p<2$ let
\[
 C_p^{\mathrm{sharp}}
 :=\sup\bigl\{\norm{H_f}_{L^{p'}(\R)}\big/\norm{f}_{L^p(\R)}\ :\
 f\in L^p(\R),\ f\ne0\bigr\}
\]
be the optimal constant in the normalization \eqref{eq:dirac-left}.
Corollary \ref{cor:Lp} gives $C_p^{\mathrm{sharp}}\le1$.  In the opposite
direction, let $G$ be a Gaussian.  The small-amplitude linearization of the
nonlinear transform gives
\[
 \lim_{\varepsilon\downarrow0}
 \frac{\norm{H_{\varepsilon G}}_{L^{p'}}}
 {\varepsilon\norm{G}_{L^p}}
 =\frac{\norm{\widehat G}_{L^{p'}}}{\norm G_{L^p}}
 =\mathbf B_p;
\]
see~\cite{KOR19} for this linearization and \cite{Beckner75} for the sharp
linear Hausdorff--Young inequality, whose extremizers include Gaussians.
Consequently,
\[
 \mathbf B_p\le C_p^{\mathrm{sharp}}\le1.
\]
Since $\mathbf B_p\to1$ as $p\uparrow2$, it follows that
$C_p^{\mathrm{sharp}}\to1$ as $p\uparrow2$.  Determining $C_p^{\mathrm{sharp}}$ for a
fixed $1<p<2$ remains a separate extremal problem; the perturbative theorem
of \cite{KOR19} shows that the nonlinear quotient is strictly smaller than
$\mathbf B_p$ in suitable small-data regimes.
\end{remark}

\subsection*{Acknowledgements}
The author is grateful to Erlan Nursultanov for many valuable discussions on Hausdorff--Young inequalities over the past four years.
The author acknowledge the use of AI tools. All mathematical arguments and proofs in the final manuscript were checked and written by the
author.

\subsection*{Funding}
This research is funded by Nazarbayev University under the grant 110326CRP0806 (D.S.).


\subsection*{Data Availability}
This manuscript has no associated data.

\end{document}